\documentclass[11pt]{article}
\usepackage{amsmath,amssymb} 
 
\def\t{\,{}^t\!}

\def\Z{{\boldsymbol{Z}}}
\def\Q{{\boldsymbol{Q}}}
\def\R{{\boldsymbol{R}}}

\newtheorem{definition}{Definition}[section]
\newtheorem{lemma}[definition]{Lemma}
\newtheorem{proposition}[definition]{Proposition}
\newtheorem{theorem}[definition]{Theorem}
\newtheorem{proof}{Proof.}

\title{On rational points of orthogonal group}
\author{Masanori Kobayashi, Chikara Nakayama}
\date{}

\begin{document}
\maketitle

\begin{abstract}
Let $n$ be a positive integer.
We show 
that a unit rational space vector whose multiple by $n$ is an integer 
vector can be extended to a rational orthonormal basis 
whose all members have the same property. 
\end{abstract}

\section{Main results}
\begin{definition}
A system $S : v_1,\ldots,v_k$ of vectors is called \textbf{orthoregular} if 
$(i)$ $v_i$ and $v_j$ are perpendicular to each other if $i \neq j$, and 
$(ii)$ $|v_i|=|v_j|\not=0$. 
We call $|v_i|$ the \textbf{length} of the system.
\end{definition}

\begin{proposition}
\label{odd}
If $n$ is odd, the length of orthoregular $n$ vectors in $\Z^n$ is an integer.
\end{proposition}
\begin{proof}
Let $v_1, \ldots,v_n$ be orthoregular integer vectors and $l$ be its length. 
Then the determinant $|v_1, \ldots, v_n| = |v_1|\cdots|v_n|=l^n$ is an integer since all the entries are integers. 
Since $l^2$ is an integer, we have $l \in \Q$, and $l \in \Z$ by the normality of $\Z$.
\end{proof}

\begin{definition}
Let $k$ be a positive integer less than $n$.
We denote the maximum number of orthoregular integer vectors extending an orthoregular system 
$S \subset \Z^n$ by $E(S)$, 
and the minimum of those $E(S)$ for all orthoregular $S$ with $\sharp S=k$ by $E(n,k)$.
For a subset $K$ of $\R$, 
we also use $E_K(n,k)$ for all orthoregular sets of $k$ integer vectors of length in $K$. 
We mainly use the cases $K=\R$ or $\Z$. 
\end{definition}

\begin{proposition}\label{p:det}
If $n$ is even or $K=\Z$, $E_K(n,n-1)=n$. 
\end{proposition}

\begin{proof}
Let $v_1,\ldots, v_{n-1}$ be an orthoregular system in $\Z^n$ of length $l$.
Take a square matrix $(v_1, \ldots, v_{n-1}, v)$, where $v$ is an arbitrary vector. 
Let $v_n$ be the vector consists of the cofactors of the last column. 
  Then $v_n$ is orthogonal to $v_1,\ldots,v_{n-1}$, thus $l^{n-1}|v_n|=|v_n|^2$ 
by cofactor expansion. 
Hence $|v_n|=l^{n-1}$. 

We will now show that $(1/l^{n-2})v_n \in \Z^n$, which will complete the proof. 
It is enough to show that the last component of $v_n$ is in $l^{n-2}\Z$ since the others are shown similarly. 
We put $v_j = 
\begin{pmatrix}
v'_j \\ x_j 
\end{pmatrix}$ for $1 \leq j \leq n-1$, where $v'_j$ has $(n-1)$-rows and $x_j$ is a scalar. 
We show that $D := |v'_1, \ldots, v'_{n-1}|$ belongs to $l^{n-2}\Z$. 
Since $l^{n-2} \in \Z$ by the assumption that $n$ is even or $l \in \Z$, 
it is enough to prove that $D^2 \in l^{2(n-2)}\Z$. 
\[
D^2 = |(v'_i,v'_j)_{1\leq i,j\leq n-1}|=|((v_i,v_j)-x_ix_j)_{i,j}|,
\]
which is nothing but $\varphi(l^2)$, where $\varphi(t)$ is the characteristic polynomial for the matrix $(x_ix_j)_{i,j}$. 
Since the matrix has rank $\leq 1$, $\varphi(t)$ has a form $t^{n-2}(t-a)$ $(a \in \Z)$. 
Thus $D^2 \in l^{2(n-2)}\Z$. 
\end{proof}

The next is the main result in this paper.

\begin{theorem}
\label{t:main}
A nonzero integer vector $v$ can be extended to an integral orthoregular basis if and only if $|v|$ is an integer.
In particular, $E_{\Z}(3,1)=3$.
\end{theorem}

\noindent 
Examples.
Here is the list of smallest primitive nonnegative integral space
vectors whose norms are integers.

$\t(0,0,1),
\t(1,2,2),
\t(0,3,4),
\t(2,3,6),
\t(1,4,8), \t(4,4,7),
\t(2,6,9), \t(6,6,7),....$

\noindent
By starting from each vector in the list,
you can always complete it to an orthoregular basis consisting of
integral vectors.
Dividing by its norm, you get a rational orthonormal basis, thus a
rational point of $O(3)$.

\section{Proof of main theorem}

In this section, we prove \ref{t:main}. 
Our proof is constructive and gives an algorithm for all solutions.

First we prove two lemmas. 

\begin{lemma}
\label{l:ortho}
  Let $v,w \in \Z^3\smallsetminus \{\bold o\}$.
  Assume $|v| \in \Z$ and $(v,w)=0$.
  Then $v_p(|w|^2)$ is even for any prime $p$ with $p\equiv3 \bmod 4$. 
  Here $(-,-)$ is the standard inner product and $v_p$ is the normalized valuation with respect to $p$. 
\end{lemma}

\begin{proof}
  Let $v=
\begin{pmatrix}
a \\ b \\ c
\end{pmatrix}, 
w=
\begin{pmatrix}
x \\ y \\ z
\end{pmatrix}$
and let $l=|v|$.
  We may assume $a\not=0$. 
  Since $ax+by+cz=0$, we have $x=\frac{-by-cz}a$.
   Hence 
$$|w|^2=\left(\frac{by+cz}a\right)^2+y^2+z^2,$$
$$a^2|w|^2=(a^2+b^2)y^2+2bcyz+(a^2+c^2)z^2.$$
Multiplying $a^2+b^2$, we deduce 
$$a^2(a^2+b^2)|w|^2=((a^2+b^2)y+bcz)^2+((a^2+c^2)(a^2+b^2)-b^2c^2)z^2.$$
  But the coefficient of $z^2$ in the right-hand-side coincides with $a^2l^2$, 
so that the both terms of the right-hand-side are perfect squares.
  Hence $v_p$ of the right-hand-side is even for any $p$
with $p\equiv3 \bmod 4$. 
  Since $v_p(a^2(a^2+b^2))$ is also even, $v_p(|w|^2)$ is even, too. 
\end{proof}

\begin{lemma}
\label{l:bilin}
   Let $a,b,c,l \in \Z$ satisfying $ac-b^2=l^2$. 
  Then there exist $x,y \in \Z$ such that $ax^2+2bxy+cy^2=l^2$ if and only if 
$a\geq0$ and, when $a>0$, $v_p(a)$ is even for any prime $p$ with $p\equiv3 \bmod 4$. 
\end{lemma}
\begin{proof}
  It is probably well-known.
  Our proof here uses the prime factorization in the integer ring $\Z[i]$, where 
$i=\sqrt{-1}$.

  We prove the only if part. 
  By multiplying $a$ with the given equation, we have $(ax+by)^2+(ly)^2=al^2$.
  Hence $l$ divides $ax+by$. 
  Letting $u=\frac{ax+by}l$, we have $u^2+y^2=a$, which implies that the desired condition holds. 

  We prove the if part.  
  The case $a=0$ is trivial. 
  Assume $a>0$. 
  If there are $u$ and $y$ with $u^2+y^2=a$ such that $a$ divides $ul-by$, then 
$x=\frac{ul-by}a$ and $y$ give an integer solution for $ax^2+2bxy+cy^2=l^2$.
  We prove the existence of $u$ and $y$ as above under the assumption that $v_p(a)$ is even for 
any prime $p$ with $p\equiv3 \bmod 4$. 
  First factorize $a$ into 
$$a=2^m
\prod_{p_j\equiv1 \bmod 4}p_j^{m_j}
\prod_{q_k\equiv3 \bmod 4}q_k^{n_k},$$
where $p_j$'s are primes with $p_j\equiv1 \bmod 4$ and 
$q_k$'s are primes with $q_k\equiv3 \bmod 4$.
  Here, each $n_k$ is even by the assumption.
  Since $l^2+b^2=ac$, we can write $l+ib$ as 
a product 
$$l+ib=u_0(1+i)^{m'}
\prod p_j^{s_j}\pi_j^{t_j}
\prod q_k^{n'_k/2}$$
in $\Z[i]$, where 
$u_0$ is a unit of $\Z[i]$, $m' \geq m$, $p_j=\pi_j\overline \pi_j$ in $\Z[i]$, 
where $\pi_j$ and $\overline \pi_j$ are conjugate prime elements, 
$2s_j + t_j \geq m_j$, 
and $n'_k \geq n_k$ is an even. 

\medskip

\noindent Claim.
  There exist $a_j \geq0, b_j \geq0$ such that $s_j + \min(t_j, b_j) \geq a_j + b_j$ 
and $2a_j +b_j = m_j$. 

\medskip

\noindent Proof of Claim. 
  If $2s_j  \leq m_j$, let $a_j=s_j$ and $b_j = m_j -2s_j$. 
  Then, we have $t_j \geq m_j -2s_j=b_j$ and $a_j$, $b_j$ satisfy the condition. 

  If $2s_j > m_j$, let $a_j := \lfloor \frac{m_j}2\rfloor$ 
and $b_j=m_j-2\lfloor \frac{m_j}2\rfloor$.
  Then, we have $b_j \leq 1$ and $s_j \geq a_j +1 \geq a_j + b_j$ and $a_j$, $b_j$ satisfy the condition. 
  This completes the proof of Claim.

\medskip

  Using these $a_j$ and $b_j$, define 
$$u+iy:=(1+i)^m\prod p_j^{a_j}\overline \pi_j^{b_j} \prod q_k^{n_k/2}.$$
  Then, since $2a_j+b_j=m_j$, we have $u^2+y^2=a$.
  Further, $ul-by$, the real part of $(u+iy)(l+ib)$, is divided by 
$2^m\prod p_j^{a_j+s_j+\min(b_j,t_j)}\prod q_k^{n_k}$.
  Since $a_j+s_j+\min(b_j,t_j) \geq 2a_j + b_j = m_j$, this is divided by $a$. 
\end{proof}

  Now we prove \ref{t:main}. 
  The only if part is by \ref{odd}.
  We prove the if part. 
  By \ref{p:det}, it is enough to show $E_{\Z}(3,1) \geq 2$. 
  Let $v_1=\begin{pmatrix}
a_1 \\ b_1 \\ c_1
\end{pmatrix}\in \Z^3\smallsetminus\{\bold o\}$ such that 
$l:=|v_1|$ is an integer.
  It suffices to show $E_{\Z}(\{v_1) \geq 2$. 
  Since $m:=(a_1,b_1,c_1)$ divides $l$, replacing $v_1$ with $\frac {v_1}m$, 
we may and will assume that $m=1$. 
  Then, the homomorphism $(v_1,-)\colon \Z^3 \to \Z$ is surjective. 
  The kernel $K$ of this surjection is free of rank two. 
  Let $Q$ be the positive definite bilinear form on $K$ obtained by 
restricting the standard metric on $\Z^3$. 
  It is enough to show that $Q$ represents $l^2$. 
  Take a base $w_1, w_2$ of $K$.
  Let $A:=$Gram($w_1,w_2$), the symmetic matrix which represents $Q$ 
with respect to $w_1, w_2$.

\medskip

\noindent 
Claim. 
$\det(A)=l^2$. 

\medskip

\noindent 
Proof of Claim. 
  Let $S$ be the area of the parallelogram spanned by $w_1, w_2$. 
  What we have to see is $S=l$. 
  Since there are isomorphisms $\Z^3/\langle v_1,w_1,w_2\rangle \cong
\Z/(v_1,v_1)\Z\cong
\Z/l^2\Z$, the volume of the parallelepiped spanned by $v_1, w_1, w_2$ 
is $l^2$. 
  On the other hand, since $v_1$ is perpendicular to $K$, 
this volume is $|v_1|S$, too.
  Hence $S=l$. 
  This completes the proof of Claim.

\medskip

  Let $a:=(w_1,w_1)$. 
  By \ref{l:ortho} with $v=v_1$ and $w=w_1$, 
  $v_p(a)$ is even for any prime $p$ with $p\equiv3 \bmod 4$.
  Hence by \ref{l:bilin} with $a=(w_1,w_1)$, $b=(w_1,w_2)$, $c=(w_1,w_2)$, 
we conclude that $Q$ represents $l^2$, which completes the proof of \ref{t:main}.

\section{More results}

Here we gather more results on $E_K(n,k)$. 

\begin{proposition}
For $0 < k < n$ and let $K$ be a subset of $\R$.
\begin{enumerate}
\item $E_{\Z}(n+1,k+1) \leq E_{\Z}(n,k)+1$. 
\item If $n$ is odd, then $E(n,k) \leq n-1$. 
\item If $n$ is odd and each coordinate of $v \in \Z^n$ is odd, then $E(v)=1$. 
\item If $n$ is odd, then $E(n,1)=1$. 
\item If $n \equiv 1 \bmod 8$, then $E_{\Z}(n,1)=1$. 
\item If $n$ is the sum of positive integers $n_i$, $E_K(n,1) \geq \min_i E(n_i,1)$. 
\item For a positive integer $m$, 
$E_K(mn,1) \geq E(n,1)$, $E_K(2m,1)\geq 2$, $E_K(4m,1)\geq 4$ and $E_K(8m,1) \geq 8$. 
\end{enumerate}
\end{proposition}

\noindent 
Question.
  Is $E_K(16m,1) \geq 16$ valid?

\begin{proof}
$(1)$ Take an orthoregular system $S:v_1, \ldots, v_k$ of  vectors which attains $E_\Z(n,k)=E(S)$. Let $l$ be the length of $S$. 
Then 
$E \left(
\begin{pmatrix}
l \\ 0 \\
\end{pmatrix}, 
\begin{pmatrix}
0 \\ v_1 \\
\end{pmatrix}, \ldots, 
\begin{pmatrix}
0 \\ v_k \\
\end{pmatrix} \right)
= E(v_1,\ldots, v_k)+1$. 

$(2)$ It is enough to show that $E(v_1)\leq n-1$ for $v_1 = \t (1 \ 1 \ 0 \ \cdots \ 0) \in \Z^n$. 
This follows from $\ref{odd}$. 
(In fact, $E(v_1)=n-1$. 
  This is seen by taking $v_2= \t (1 \ -1 \ 0 \ \cdots \ 0 ), \ldots, v_{2k-1}=\t (0 \ \cdots \ 0 \ 1 \ 1 \ 0 \ \cdots \ 0 ), v_{2k}=\t (0 \ \cdots \ 0 \ 1 \ -1 \ 0 \ \cdots \ 0),\ldots$ $(1 \leq k \leq (n-1)/2)$.
Here the $i$-th components of $v_{2k-1}$ and of $v_{2k}$ are zero unless $i=2k-1,2k$.) 

$(3)$ Suppose that an integer vector $v' = \t ( a_1 \ \cdots \ a_n)$ satisfies $(v,v')=0$ and $|v|=|v'|$. 
By taking modulo two, we have $|v|^2=|v'|^2=a_1^2+\cdots+a_n^2 \equiv a_1 + \cdots + a_n$. 
Since the components of $v$ are odd, the last integer is congruent to $(v,v')=0$, 
that is, $|v|^2 \equiv0 \bmod 2$. 
  On the other hand, we have $|v|^2 \equiv n \equiv 1 \bmod 2$. 
A contradiction.

$(4)$ Take $v = \t(1 \cdots 1)$ and apply $(3)$. 

$(5)$ Use $(3)$ and the following Claim.
  (The part $(b) \Rightarrow (a)$ suffices.)

\medskip

\noindent 
Claim. For an odd integer $n>0$, the followings are equivalent. 

$(a)$ There exist $n$ odd integers whose square sum is again square. 

$(b)$ $n \equiv 1 \mod 8$. 

\medskip

\noindent 
Proof of Claim. 
$(a) \Rightarrow (b)$. 
  Assume that there are $x_1,\ldots,x_n,y \in \Z$ such that 
$x_1^2+\cdots+ x_n^2=y^2$. 
  Since the square of any odd number is congruent to $1$ modulo $8$, 
we have $n \equiv x_1^2+\cdots+ x_n^2=y^2 \bmod 8$.
  Since $y$ is also odd, this implies $n \equiv 1 \bmod 8$. 

$(b) \Rightarrow (a)$. 
  Take $m \geq1$ such that $(2m-1)^2 \leq n \leq (2m+1)^2$. 
  If we prove $n \geq \frac {(2m+1)^2-n}8$, 
then, $(x_1,\ldots, x_n)=(3,\ldots,3,1\ldots,1)$ will do, 
where the number of $3$ is $\frac {(2m+1)^2-n}8$, 
for $x_1^2+\cdots+ x_n^2=n+(9-1)\frac {(2m+1)^2-n}8=(2m+1)^2$.

  We prove the above inequality.  
  This is equivalent to $9n \geq (2m+1)^2$, which reduces to 
$9(2m-1)^2 \geq (2m+1)^2$, $3(2m-1) \geq 2m+1$, $m\geq1$, 
which completes the proof of Claim.

$(6)$ Let $v \in \Z^n\smallsetminus\{\bold o\}$. 
  By the isomorphism $\Z^n \cong \Z^{n_1} \oplus \Z^{n_2}\oplus \cdots$, we regrad $v$ as 
a collection $v_1, v_2,\ldots$ with $v_i \in \Z^{n_i}$. 
  Let $m:=\min_iE(n_i,1)$.
  For each $i$, if $v_i$ is not zero, 
apply $E(n_i,1)$ to $v_i$, and we get an orthoregular set $v_{i1}=v_i,v_{i2},\ldots, v_{im}$. 
  If $v_i$ is zero, let $v_{i1}=v_{i2}=\cdots=v_{im}$ be zero. 
  For each $j=2,\ldots, m$, let $v_j \in \Z^n$ be the vector corresponding to $(v_{ij})_i$.
  Then, $v_1,v_2,\ldots, v_m$ is the desired system. 

$(7)$ By $(6)$, we may assume that $m=1$.  
  Then this is by the next proposition with $K=\R$.
\end{proof}

\begin{proposition}
Let $K$ be an arbitrary subset of $\R$ satisfying $-K=K$. 
If $n=2,4$ or $8$, then $E_K(n,1)=n$ for any $K$. 
\end{proposition}
\begin{proof}
  Let $v_1 \in K^n$. 
  Via a standard matrix representation of complex number field, quarternion algebra or Cayley algebra as an $\R$-algebra, 
the rotations multiplied by the unit imaginary numbers gives the desired $v_2,\ldots, v_n$. 
  Explicitly, for 
$v_1=
\begin{pmatrix}
a_0 \\ a_1 \\
\end{pmatrix}$, 
we give 
$v_2=iv_1=
\begin{pmatrix}
-a_1 \\ a_0 \\
\end{pmatrix};$ for 
$v_1=
\begin{pmatrix}
a_0 \\ a_1 \\ a_2 \\ a_3
\end{pmatrix}$, 
we give 
$(v_2,v_3,v_4)=(iv_1,jv_1,kv_1)=
\left(
\begin{pmatrix}
-a_1 \\ a_0 \\ -a_3 \\ a_2 
\end{pmatrix}, 
\begin{pmatrix}
-a_2 \\ a_3 \\ a_0 \\ -a_1 
\end{pmatrix}, 
\begin{pmatrix}
-a_3 \\ -a_2 \\ a_1 \\ a_0 
\end{pmatrix}
\right);$ for 
$v_1=
\begin{pmatrix}
a_0 \\ a_1 \\ a_2 \\ a_3 \\ a_4 \\ a_5 \\ a_6 \\ a_7
\end{pmatrix}$, 
we give 
$$(v_2,v_3,v_4,v_5,v_6,v_7,v_8)=
\left(
\begin{pmatrix}
-a_1 \\ a_0 \\ -a_3 \\ a_2 \\ -a_5 \\ a_4 \\ a_7 \\ -a_6
\end{pmatrix}, 
\begin{pmatrix}
-a_2 \\ a_3 \\ a_0 \\ -a_1 \\ -a_6 \\ -a_7 \\ a_4 \\ a_5
\end{pmatrix}, 
\begin{pmatrix}
-a_3 \\ -a_2 \\ a_1 \\ a_0 \\ -a_7 \\ a_6 \\ -a_5 \\ a_4 
\end{pmatrix},
\begin{pmatrix}
-a_4 \\ a_5 \\ a_6 \\ a_7 \\ a_0 \\ -a_1 \\ -a_2 \\ -a_3 
\end{pmatrix},
\begin{pmatrix}
-a_5 \\ -a_4 \\ a_7 \\ -a_6 \\ a_1 \\ a_0 \\ a_3 \\ -a_2 
\end{pmatrix},
\begin{pmatrix}
-a_6 \\ -a_7 \\ -a_4 \\ a_5 \\ a_2 \\ -a_3 \\ a_0 \\ a_1
\end{pmatrix},
\begin{pmatrix}
-a_7 \\ a_6 \\ -a_5 \\ -a_4 \\ a_3 \\ a_2 \\ -a_1 \\ a_0 
\end{pmatrix}
\right).$$
\end{proof}

\begin{proposition}
$E(6,k) \leq 4$ for $1 \leq k \leq 4$. 
\end{proposition}

\begin{proof}
  Let $v_1=\t (1,1,1,1,1,1)$.
  It is enough to show $E(6,v_1)\leq 4$. 
  First observe that any integer vector who has the same length as $v_1$ satisfies 
either (a) the absolute value of any its component is $1$ or 
(b) the absolute values of its components are $2,1,1,0,0,0$. 

  If an integer vector satisfying (a) is perpendicular to $v_1$, then it has exactly 
$3$ components whose value is $-1$. 
  From this, we see that there is no orthoregular set $v_1, v_2, v_3$ such that 
both $v_2$ and $v_3$ satisfy (a). 

  On the other hand, if an integer vector satisfying (b) is perpendicular to $v_1$, 
then it satisfies 
(b${}'$) its components are $2,-1,-1,0,0,0$ or $-2,1,1,0,0,0$. 
Further, if both $v_2$ and $v_3$ satisfy (b${}'$) and they are perpendicular, 
their supports should be disjoint.
  From this, we see that there is no 
orthoregular set $v_1, v_2, v_3, v_4$ such that 
three of $v_2$, $v_3$, and $v_4$ satisfy (b). 

  Therefore there are no orthoregular set $v_1, v_2, v_3, v_4, v_5$, that is, 
$E(6,v_1)\leq 4$, as desired. 
\end{proof}

\noindent 
Remark: In fact, $E(6,v_1)=4$ as seen by 
$v_2=\t(1,1,1,-1,-1,-1), 
v_3=\t(2,-1,-1,0,0,0), 
v_4=\t(0,0,0,2,-1,-1)$.

\bigskip

\noindent Masanori Kobayashi

\noindent Department of Mathematics and Information Sciences \\ 
Tokyo Metropolitan University \\
1-1 Minami-Ohsawa, Hachioji, Tokyo, 192-0397 \\
Japan

\noindent kobayashi-masanori@tmu.ac.jp

\bigskip

\noindent Chikara Nakayama

\noindent Department of Economics \\ Hitotsubashi University \\
2-1 Naka, Kunitachi, Tokyo 186-8601 \\ Japan

\noindent c.nakayama@r.hit-u.ac.jp
\end{document}